\documentclass[11pt,a4paper]{article}

\usepackage[utf8]{inputenc}
\usepackage{amsmath, amssymb, amsthm}
\usepackage{enumitem}
\usepackage{geometry}
\newtheorem{theorem}{Theorem}[section]
\newtheorem{lemma}[theorem]{Lemma}
\newtheorem{definition}[theorem]{Definition}
\newtheorem{corollary}[theorem]{Corollary}
\theoremstyle{remark}
\newtheorem{remark}[theorem]{Remark}

\title{\textbf{On the Uniqueness of Best Non-decreasing Approximation in Orlicz Spaces}}
\author{A. Benavente $^1$, J. Costa Ponce $^2$, S. Favier $^3$}
\date{\today}

\begin{document}

\maketitle

\begin{abstract}
Given an approximately continuous function $f$ in an Orlicz space  $L^\Phi([a,b]),$ for a suitable class of convex functions $\Phi,$ we employ a characterization of the best non-decreasing approximation set to establish its continuity, which in turn yields the  uniqueness property for the best non-decreasing approximation in $L^\Phi([a,b]).$

\end{abstract}

\footnotetext[3]{\textit{Research partially supported by CONICET and Universidad Nacional de San Luis.}}

\vspace{1cm}

\section{Introduction}

The study of properties such as existence and uniqueness,  for a best
approximation to a given function from a specified class is a
central topic in Approximation Theory. When the approximation class
has a lattice structure, such as the set of non-decreasing
functions, the problem poses particular challenges. In their seminal
work, Landers and Rogge \cite{landersrogge} established results on
the existence of best approximants in $L^\Phi$-spaces from function
lattices. More examples of best approximation, where the uniqueness
property is still an open problem, are treated in \cite{Levis}.

A case of special interest is the approximation of a function $f \in
L^1([0,1]^n)$ by a non-decreasing function. Darst and Huotari
\cite{darsthuotari},  followed by Z\'o and Iturrieta
\cite{zoiturrieta} in Orlicz spaces, among others, dealt  with
this problem, proving uniqueness under the  assumption that $f$ was
a continuous function and then $f \in L^\infty([0,1]^n)$. Also this
topic was considered by Darst and Fu \cite{darstfu} for an
approximately continuous function $f\in L^1 ((0,1)^n). $ A
significant advance was made by Smith and Swetits
\cite{smithswetits}, who managed to prove uniqueness while removing the
boundedness hypothesis over $f$.

In order to obtain uniqueness of best non decreasing approximation  for an approximately continuous function $f$ that is not necessarily bounded and for a suitable convex function $\Phi$, this paper applies the method of Smith and Swetits in the broader framework of Orlicz spaces $L^{\Phi}([a, b]).$  This is done without using the Radon-Nikodym set up  for the characterization of the best approximations introduced by \cite{Brunk} and considered in \cite{Cuenya} or the requirement that the greatest and least best non-decreasing approximations be used, as in \cite{Marano}.
\section{Definitions and notation}

Throughout this paper we let    $([a,b], \mathcal{B}, \mu)$  denote  the finite measure space, where $\mathcal{B}$ is the Borel $\sigma$-algebra over $[a,b]$ and $\mu$ is the standard Lebesgue measure. We recall that $\mathcal{M}=\mathcal{M}([a,b],\mathcal{B}, \mu)$ is the set of all $\mu$-equivalent classes of $\mathcal{B}$-measurable functions on $[a,b]$.

We begin by defining the specific class of convex functions used to construct the Orlicz spaces considered in  this work.   A function $\varphi:[0,\infty)\to[0,\infty)$ belongs to the class $\mathcal{F}$ if it is non-decreasing, continuous, positive, satisfies $\varphi(0^{+})=0$ and $\lim_{t\to\infty}\varphi(t)=\infty$.

Analogously, we say a function $\varphi:[0,\infty)\to[0,\infty)$ belongs to the class $\mathcal{F_{\infty}}$ if $\varphi$ is bounded, non-decreasing, continuous, positive for $t>0$, and satisfies $\varphi(0^{+})=0$.

An $N$-function $\Phi$ is a convex function represented by $\Phi (x) = \int_{0}^{x}\varphi (t)dt$ for some $\varphi\in\mathcal{F}$ and similarly,
an $N_\infty$-function $\Phi$ is a convex function represented by $\Phi (x) = \int_{0}^{x}\varphi (t)dt$ for some $\varphi\in\mathcal{F}_{\infty}$.
Therefore, $N$-functions and $N_\infty$-functions are differentiable on $[0,\infty)$.

In addition, we assume that $N$-functions and $N_\infty$-functions satisfy the $\Delta_2$ condition: there is a constant $K_2 > 0$ such that $\Phi(2x) \le K_2\Phi(x)$ for all $x \ge 0$. A comprehensive treatment of this condition can be found in \cite{krasnoselskii}.

The Orlicz space $L^{\Phi}([a,b])=L^{\Phi}$ is the set of all equivalence classes of measurable functions $f$ for which $\int_a^b \Phi(|f(x)|)d\mu < \infty$. If $\varphi\in\mathcal{F}$ or $\varphi\in\mathcal{F_{\infty}}$, the space $L^{\varphi}([a,b])=L^{\varphi}$ is defined in a similar way and in these cases, $L^{\Phi}\subset L^{\varphi}$.
 
Given a subset $C \subset L^\Phi$, an element $g^* \in C$ is a best $\Phi$-approximation to $f \in L^\Phi$ from $C$ if
$$ \int_{a}^{b}\Phi(|f-g^*|)d\mu = \inf_{g \in C}\int_{a}^{b}\Phi(|f-g|)d\mu. $$
We denote the set of such best approximants by $\mu_\Phi(f/C),$ and the subset $C$ will be called the approximation class in this case. We naturally assume that $f\notin C$, for the other case is trivial.

A set $C \subset \mathcal{M}$ is a lattice if for any $f,g \in C$, it holds that $f \wedge g := \min(f,g) \in C$ and $f \vee g := \max(f,g) \in C$.
 Furthermore, a lattice $C \subset \mathcal{M}$ is $\Phi$-closed if for any sequence $\{f_n\} \subset C$ such that $f_n \uparrow f$ or $f_n \downarrow f$ for some $f \in L^\Phi$, it holds that $f \in C$. We use $f_n \uparrow f$ (or $f_n \downarrow f$) to denote that the sequence converges non-decreasingly (or non-increasingly) to $f$.

An interesting way of generating $\Phi$-closed lattices in $L^\Phi([a,b])$ is by considering a $\sigma$-lattice $\mathcal{L}\subset\mathcal{B}$, which means that $\mathcal{L}$ is closed for countable unions, countable intersections, and $\emptyset,[a,b]\in \mathcal{L}$. A function $f$ is said to be $\mathcal{L}$-measurable if there exists an $\mathcal{B}$-measurable function $g$ such that $f = g$ almost everywhere on $[a, b]$, and $\{g > c\} \in \mathcal{L}$ for all $c \in \mathbb{R}$. The set
$ L^\Phi(\mathcal{L}) := \{f \in L^\Phi([a, b]) : f \text{ is } \mathcal{L}\text{-measurable} \} $
turns out to be a $\Phi$-closed lattice.

In this paper, we specifically focus on the case where $ C = M_{\Phi},$ which is the $\Phi$-closed lattice of all non-decreasing functions in $ L^{\Phi}([a, b]) $. We note that $ \mu_{\Phi} (f / M_{\Phi}) \neq \emptyset, $ according to \cite{landersrogge}, for any $ f \in L^{\Phi}([a, b])$ with $\Phi$ a convex function satisfying $\Phi(0)=0$, $\lim_{t\to\infty}\Phi(t)=\infty$ and the $\Delta_2$ condition. Additionally, $M_{\Phi}$ is exactly the $\Phi$-closed lattice in $L^{\Phi}([a, b])$ generated by the $\sigma$-lattice $\mathcal{D}\subset \mathcal{B}$, where $\mathcal{D} = \big\{ (c, b], \, [c, b] : c \in [a, b] \big\} \cup \{\emptyset\}$.

\section{Characterization and Fundamental Properties}

Our analysis relies upon a characterization theorem for best approximants from a function lattice, an already existing result for Orlicz spaces (see \cite{favierzo}), which we reworked for this paper.

\begin{theorem}[Characterization of Best Approximation from a Lattice]
\label{thm:lattice_characterization}
Let $\Phi$ be a $N$ or a $N_\infty$ function that satifies the $\Delta_2$ condition. Let $\mathcal{L}$ be a
$\sigma$-lattice, $\mathcal{L} \subseteq \mathcal{B}$. Given $f \in L^{\Phi}$, an element $g^* \in L^{\Phi}(\mathcal{L})$ is a best approximant to $f$ from $L^{\Phi}(\mathcal{L})$ if and only if
\begin{equation}
\int_a^b \varphi(|f-g^{*}|)sgn(f-g^{*})(g^{*}-g)d\mu \ge 0 \tag{I}\label{carac}
\end{equation}
for all $g \in L^{\Phi}(\mathcal{L})$.
\end{theorem}

\begin{proof}
Let $g^* \in \mu_\Phi(f/L^\Phi(\mathcal{L}))$ and consider any other $g \in L^\Phi(\mathcal{L})$ such that $g \neq g^*$. We define the auxiliary function $F_g(\epsilon)$ for $\epsilon \ge 0$:
$$
F_{g}(\epsilon):=\int_{a}^{b} \Phi(|f-(\epsilon g+(1-\epsilon)g^{*})|) \,d\mu.
$$
First, we observe that $F_g$ is a convex function on $[0, \infty)$. Indeed, for any $a_1, b_1, \epsilon_1, \epsilon_2 \ge 0$ with $a_1+b_1=1$, the convexity of $\Phi$ and the properties of the absolute value and the integral yield:
\begin{align*}
F_{g}(a_1 \epsilon_{1}+b_1 \epsilon_{2}) &= \int_a^b \Phi(|f-((a_1 \epsilon_{1}+b_1 \epsilon_{2})g+(1-(a_1 \epsilon_{1}+b_1 \epsilon_{2}))g^{*})|) \,d\mu \\
& = \int_a^b \Phi(|a_1(f-(\epsilon_1 g + (1-\epsilon_1)g^*)) + b_1(f-(\epsilon_2 g + (1-\epsilon_2)g^*))|) \,d\mu \\
& \le \int_a^b \Phi(a_1|f-(\epsilon_1 g + (1-\epsilon_1)g^*)| + b_1|f-(\epsilon_2 g + (1-\epsilon_2)g^*)|) \,d\mu \\
& \le a_1 \int_a^b \Phi(|f-(\epsilon_1 g+(1-\epsilon_1)g^{*})|) \,d\mu+b_1 \int_a^b \Phi(|f-(\epsilon_2 g+(1-\epsilon_2)g^{*})|) \,d\mu \\
& = a_1 F_{g}(\epsilon_{1})+b_1 F_{g}(\epsilon_{2}).
\end{align*}
Since $g^*$ is a best approximant, $F_g$ must attain its global minimum at $\epsilon=0$. Due to the convexity of $F_g$, this is equivalent to $F_{g}'(0^+) \ge 0$. The derivative is given by:
$$
F_{g}'(0^+) = \lim_{\epsilon \to 0^{+}} \frac{F_{g}(\epsilon)-F_{g}(0)}{\epsilon} = \lim_{\epsilon \to 0^{+}} \frac{1}{\epsilon}\left\{\int_a^b \Phi(|f-(\epsilon g+(1-\epsilon)g^{*})|) \,d\mu-\int_a^b \Phi(|f-g^{*}|) \,d\mu\right\}.
$$
To justify passing the limit inside the integral, we must find an integrable function that dominates the difference quotient. By the Mean Value Theorem and the convexity of $\Phi$, for $u, v \ge 0$, we have $|\Phi(u) - \Phi(v)| \le |u-v|\varphi(\max(u,v))$. Let $u = |f-g^*|$ and $v = |f-(\epsilon g + (1-\epsilon)g^*)|$. Then $|u-v| \le \epsilon|g-g^*|$. Thus,
$$
\frac{|\Phi(v) - \Phi(u)|}{\epsilon} \le \frac{\epsilon|g-g^*|}{\epsilon} \varphi(\max(u,v)) \le |g-g^*|\varphi(|f-g^*| + |g-g^*|).
$$ 

For any $w \geqslant 0$, we have:
$$ \Phi(2w) = \int_{0}^{2w}\varphi(t)dt \geqslant \int_{w}^{2w}\varphi(t)dt \geqslant \int_{w}^{2w}\varphi(w)dt = w\varphi(w). $$

Taking $w=|f-g^*| + |g-g^*|$, we obtain:
$$ (|f-g^*| + |g-g^*|)\varphi(|f-g^*| + |g-g^*|) \leqslant \Phi(2(|f-g^*| + |g-g^*|)). $$

Thus,
$$ |g-g^*|\varphi(|f-g^*| + |g-g^*|) \leqslant (|f-g^*| + |g-g^*|)\varphi(|f-g^*| + |g-g^*|) \leqslant \Phi(2(|f-g^*| + |g-g^*|)). $$

Since $\Phi$ satisfies the $\Delta_2$ condition,
$$ |g-g^*|\varphi(|f-g^*| + |g-g^*|) \leqslant K_2 \Phi(|f-g^*| + |g-g^*|). $$

We note that $L^\Phi$ is a vector space and $f,g,g^* \in L^\Phi$, so we have $|f-g^*| + |g-g^*|\in L^\Phi$.
Therefore the function $|g-g^*|\varphi(|f-g^*| + |g-g^*|)$ is an integrable function over $[a,b]$, so we apply Lebesgue's Dominated Convergence Theorem to obtain

\begin{align*}
F_{g}'(0^+) &= \int_a^b \left. \frac{d}{d\epsilon} \Phi(|f-(\epsilon g+(1-\epsilon)g^{*})|) \right|_{\epsilon=0} d\mu \\
&= \int_a^b \varphi(|f-g^{*}|) sgn(f-g^{*}) (g^*-g) d\mu.
\end{align*}
Imposing the condition $F_{g}'(0^+) \ge 0$ yields the integral inequality $\eqref{carac}$ stated in the theorem. Conversely, if $\eqref{carac}$ holds for all $g \in L^\Phi(\mathcal{L})$, then $F_{g}'(0^+) \ge 0$. By the convexity of $F_g$, this guarantees that $\epsilon=0$ is a global minimum, confirming that $g^*$ is a best approximant. This completes the proof.
\end{proof}

From this characterization, we define two auxiliary functions, analogous to those introduced by Smith and Swetits.

\begin{definition}
Let $g^* \in \mu_\Phi(f/M_\Phi)$. We define
 $$\Phi_{g^{*}}:=\varphi(|f-g^{*}|)sgn(f-g^{*})$$
 and
 $$r_{g^{*}}(c):=\int_{a}^{c}\Phi_{g^{*}}d\mu, \quad a \le c \le b$$
\end{definition}

Later on, the following auxiliary lemmas will be of great importance.

\begin{lemma}\label{lem:min_at_endpoints}
	Let $f:[a,b]\rightarrow\mathbb{R}$ be a continuous function such that $f(a)=f(b)=0$ and $f(x)>0$ for $x \in (a,b)$. Then there exists a sequence $\{\epsilon_n\}_{n=1}^\infty$ tending to $0$ such that for each $n$, $f$ is non-constant and takes its minimum over $[a+\epsilon_n, b-\epsilon_n]$ at $a+\epsilon_n$ or at $b-\epsilon_n$.
\end{lemma}
\begin{proof}
	Let $\{\alpha_n\}_{n=1}^\infty$ be a decreasing positive sequence tending to 0. We construct the sequence $\{\epsilon_n\}_{n=1}^\infty$ inductively.
	
	For $n=1$, since $\alpha_n \to 0$ and $f$ is continuous with $f(a)=0$, there exists an index $k_1$ such that $f(a+\alpha_{k_1}) < f\left(\frac{a+b}{2}\right)$. Define 
	$ \Lambda_1 := \min\left\{x \in [a+\alpha_{k_1}, b-\alpha_{k_1}] : f(x) = \min_{y \in [a+\alpha_{k_1}, b-\alpha_{k_1}]} f(y)\right\}. $
	If $\Lambda_1 > \frac{a+b}{2}$, set $b-\epsilon_1 := \Lambda_1$. If $\Lambda_1 \le \frac{a+b}{2}$, set $a+\epsilon_1 := \Lambda_1$. On the resulting interval $[a+\epsilon_1, b-\epsilon_1]$, $f$ is non-constant and attains its minimum at an endpoint.
	
	Inductively, suppose at step $n-1$ we have an interval $[a+\epsilon_{n-1}, b-\epsilon_{n-1}]$ as described in the lemma. Since $\alpha_n \to 0$, there exists an index $k_n$ such that $f(a+\alpha_{k_n}) < \min\{f(a+\epsilon_{n-1}), f(b-\epsilon_{n-1})\}$. Define 
	$ \Lambda_n := \min\left\{x \in [a+\alpha_{k_n}, b-\alpha_{k_n}] : f(x) = \min_{y \in [a+\alpha_{k_n}, b-\alpha_{k_n}]} f(y)\right\}. $
	We set $a+\epsilon_n$ or $b-\epsilon_n$ to be $\Lambda_n$ depending on whether $\Lambda_n \le \frac{a+b}{2}$ or $\Lambda_n > \frac{a+b}{2}$. By construction, $\epsilon_n < \epsilon_{n-1}$, and since $\alpha_{k_n} \le \epsilon_n < \epsilon_{n-1} \le \alpha_{k_{n-1}}$, the sequence $\{\epsilon_n\}_{n=1}^\infty$ is strictly decreasing and converges to 0.
\end{proof}

\begin{lemma}\label{lem:swetits_analogue}
The following properties hold for any $g^* \in \mu_\Phi(f/M_\Phi).$
\begin{enumerate}
    \item[\bfseries 1.] $\int_{a}^{b}\Phi_{g^{*}}g^{*}d\mu=0$.
    \item[\bfseries 2.] $r_{g^{*}}(c)\ge 0$ for all $c\in[a,b]$.
    \item[\bfseries 3.] $\int_{a}^{b}\Phi_{g^{*}}d\mu=0$.
    \item[\bfseries 4.] $\int_{c}^{b}\Phi_{g^{*}}d\mu\le 0$ for all $c\in[a,b]$.
    \item[\bfseries 5.] If $g^{*}$ has a jump discontinuity at $c \in (a,b)$, then $r_{g^{*}}(c)=0$.
    \item[\bfseries 6.] If $r_{g^{*}}(c)>0$ for some $c\in(a,b)$, then $g^{*}$ is constant in a neighborhood of $c$.
\end{enumerate}
\end{lemma}

\begin{proof}
\begin{itemize}
    \item[\bfseries 1.] From Theorem \ref{thm:lattice_characterization}, we have that $\int_a^b \Phi_{g^*} (g^*-g) d\mu \ge 0$ for all $g \in M_\Phi$. Choosing $g=2g^*$ gives $\int_a^b \Phi_{g^*} (-g^*) d\mu \ge 0$, which implies $\int_a^b \Phi_{g^*} g^* d\mu \le 0$. Choosing $g = \frac{1}{2}g^*$ gives $\int_a^b \Phi_{g^*} (\frac{1}{2}g^*) d\mu \ge 0$, which implies $\int_a^b \Phi_{g^*} g^* d\mu \ge 0$. The result follows.

    \item[\bfseries 2.] Fix $c \in [a,b]$. Define the test function $g(x):=-1$ for $a \le x \le c$ and $g(x):=0$ for $c < x \le b$. This function belongs to $M_\Phi$. From the characterization theorem and item 1, we have $0 = \int_a^b \Phi_{g^*}g^* d\mu \ge \int_a^b \Phi_{g^*}g d\mu$. The second integral evaluates to $\int_a^c \Phi_{g^*} (-1) d\mu + \int_c^b \Phi_{g^*} (0) d\mu = -\int_a^c \Phi_{g^*} d\mu = -r_{g^*}(c)$. Thus, $0 \ge -r_{g^*}(c)$, which implies $r_{g^*}(c) \ge 0$.

    \item[\bfseries 3.] Using the characterization $\int_a^b \Phi_{g^*}g^* d\mu \ge \int_a^b \Phi_{g^*}g d\mu$ and item 1, we choose the constant test functions $g=1$ and $g=-1$. For $g=1$, we get $0 \ge \int_a^b \Phi_{g^*} d\mu$. For $g=-1$, we get $0 \ge -\int_a^b \Phi_{g^*} d\mu$. Together, these imply $\int_a^b \Phi_{g^*} d\mu = 0$.

    \item[\bfseries 4.] This is a direct consequence of items 2 and 3, since $\int_c^b \Phi_{g^*} d\mu = \int_a^b \Phi_{g^*} d\mu - \int_a^c \Phi_{g^*} d\mu = 0 - r_{g^*}(c) \le 0$.

    \item[\bfseries 5.]
Let $\epsilon>0$ be such that the jump at $c$ satisfies $a+\epsilon<c<b-\epsilon$. Let $g$ be a non-decreasing function such that $g(x)=g^{*}(x)$ for all $x\in [a,a+\epsilon]\cup[b-\epsilon,b]$.

Applying the integration-by-parts formula (available in \cite{evansgariepy}), we have
$$r_{g^{*}}g^{*}|_{a+\epsilon}^{b-\epsilon}=\int_{a+\epsilon}^{b-\epsilon}r_{g^{*}}dg^{*}+\int_{a+\epsilon}^{b-\epsilon}\Phi_{g^{*}}g^{*}d\mu$$
$$r_{g^{*}}g|_{a+\epsilon}^{b-\epsilon}=\int_{a+\epsilon}^{b-\epsilon}r_{g^{*}}dg+\int_{a+\epsilon}^{b-\epsilon}\Phi_{g^{*}}gd\mu$$

By construction, $g$ and $g^*$ coincide at the endpoints $a+\epsilon$ and $b-\epsilon$. Furthermore, by Theorem \ref{thm:lattice_characterization}, we know that $\int_{a+\epsilon}^{b-\epsilon}\Phi_{g^{*}}g^{*}d\mu\geqslant\int_{a+\epsilon}^{b-\epsilon}\Phi_{g^{*}}gd\mu$. From these facts, it follows

\begin{equation}
    0\leqslant\int_{a+\epsilon}^{b-\epsilon}r_{g^{*}}dg^{*}\leqslant\int_{a+\epsilon}^{b-\epsilon}r_{g^{*}}dg.
\tag{II}\label{cota}\end{equation}

Now we consider first the case where there exists some $\overline{c} \in (a+\varepsilon, b-\varepsilon)$ such that $r_{g^{*}}(\overline{c})=0.$

Now we set the following non decreasing function  $$g_{1}(x):=\left\{ \begin{array}{ll}
g^{*}(x) & \text{if } a\leq x \leq a+\epsilon \\
g^{*}(a+\epsilon) & \text{if } a+\epsilon < x < \overline{c} \\
g^{*}(b-\epsilon) & \text{if } \overline{c} \leq x < b-\epsilon \\
g^{*}(x) & \text{if } b-\epsilon \leq x \leq b
\end{array}\right.$$
Then we obtain
$$\int_{a+\epsilon}^{b-\epsilon}r_{g^{*}}dg_1=r_{g^{*}}(\overline{c}) \, [g^{*}(b-\epsilon)-g^{*}(a+\epsilon)]=0.$$
Thus we have the following chain of inequalities
$$r_{g^{*}}(c)[g^{*}(c^{+})-g^{*}(c^{-})]\leqslant\int_{a+\epsilon}^{b-\epsilon}r_{g^{*}}dg^{*}\leqslant\int_{a+\epsilon}^{b-\epsilon}r_{g^{*}}dg_1=0.$$

From this and $\eqref{cota}$, we conclude that $r_{g^{*}}(c)=0$. This proves the assertion under the additional assumption that $r_{g^{*}}(\overline{c})=0$ for some $\overline{c}\in(a+\epsilon,b-\epsilon)$.

Now, we consider the case where $r_{g^{*}}(x)>0$ for all $x\in(a,b)$. In this scenario, we can apply Lemma \ref{lem:min_at_endpoints} considering  a sequence $\epsilon_{n} \to 0$ such that $r_{g^{*}}$ attains its minimum over $[a+\epsilon_{n},b-\epsilon_{n}]$ at either $a+\epsilon_{n}$ or $b-\epsilon_{n}$. Assume the minimum is achieved at $a+\epsilon_{n}$. We define the test function $g_2$ as such
$$g_{2}(x):=\left\{ \begin{array}{ll}
g^{*}(x) & \text{if } a\leq x \leq a+\epsilon_n \\
g^{*}(b-\epsilon_{n}) & \text{if } a+\epsilon_n < x < b-\epsilon_{n} \\
g^{*}(x) & \text{if } b-\epsilon_n \leq x \leq b
\end{array}\right.$$
As before, we can compute the integral of $r_{g^{*}}$ with respect to $g_{2}$, which yields the inequality chain
$$\int_{a+\epsilon_n}^{b-\epsilon_n}r_{g^{*}}dg_2=r_{g^{*}}(a+\epsilon_n)[g^{*}(b-\epsilon_n)-g^{*}(a+\epsilon_n)]\leqslant\int_{a+\epsilon_n}^{b-\epsilon_n}r_{g^{*}}dg^{*}$$
So it must be
$$ r_{g^{*}}(a+\epsilon_n)[g^{*}(b-\epsilon_n)-g^{*}(a+\epsilon_n)]=\int_{a+\epsilon_n}^{b-\epsilon_n}r_{g^{*}}dg^{*}$$

This can be rewritten as $$\int_{a+\epsilon_n}^{b-\epsilon_n}[r_{g^{*}}(\cdot)-r_{g^{*}}(a+\epsilon_n)]dg^{*}=0$$
If $g^{*}$ had a jump at $c\in(a+\epsilon_n,b-\epsilon_n)$ with $n$ large enough, then we can once again argue that $$[r_{g^{*}}(c)-r_{g^{*}}(a+\epsilon_n)][g^{*}(c^{+})-g^{*}(c^{-})]\leqslant\int_{a+\epsilon_n}^{b-\epsilon_n}[r_{g^{*}}(\cdot)-r_{g^{*}}(a+\epsilon_n)]dg^{*}=0$$
So we claim that $r_{g^{*}}(a+\epsilon_n)=r_{g^{*}}(c)$.

We take a decreasing subsequence $\epsilon_{n_k}$ such that $r_{g^{*}}$ always attains its minimum at the left end of $[a+\epsilon_{n_k},b-\epsilon_{n_k}]$, and consider the limit as $k\rightarrow\infty$. Using the continuity of $r_{g^{*}}$, we see that  $r_{g^{*}}(c)=r_{g^{*}}(a^{+})=0$.

A symmetric argument, assuming the minimum is attained at the right endpoint, shows that $r_{g^{*}}(c)=r_{g^{*}}(b^{-})=0$.

But this contradicts our assumption $r_{g^{*}}(x)>0$ for all $x\in(a,b)$, so in this case $g^{*}$ cannot jump at any $c\in(a,b)$.

Since $g^{*}$ is monotone, we have also proved that if $r_{g^{*}}(c)>0$, then $g^{*}$ is continuous at $c$.

\item[\bfseries 6.]

Let $\epsilon>0$ be such that $c$ satisfies $a+\epsilon<c<b-\epsilon$.

First, suppose that $r_{g^{*}}(\overline{c})=0$ for some $\overline{c}\in(a+\epsilon,b-\epsilon)$.
We define the test function $g_{3}$ as
$$ g_{3}(x):=\left\{ \begin{array}{ll} g^{*}(a+\epsilon) & \text{if } a+\epsilon\leq x \leq \overline{c} \\ g^{*}(b-\epsilon) & \text{if } \overline{c} < x \leq b-\epsilon \end{array}\right. $$
Since $r_{g^{*}}$ is continuous and $r_{g^{*}}(c)>0$, there exist $x_1,x_2$ such that $a+\epsilon \leqslant x_{1} < c < x_{2} \leqslant b-\epsilon$ and $\min_{x\in[x_1,x_2]}r_{g^{*}}(x)>0$.
Since $g^*$ is non-decreasing and $r_{g^*} \geqslant 0$, we can observe the following chain of inequalities:
$$ \min_{x\in[x_1,x_2]}r_{g^{*}}(x)[g^{*}(x_2)-g^{*}(x_1)] \leqslant \int_{x_1}^{x_2}r_{g^{*}}dg^{*} \leqslant \int_{a+\epsilon}^{b-\epsilon}r_{g^{*}}dg^{*} \leqslant \int_{a+\epsilon}^{b-\epsilon}r_{g^{*}}dg_3 = $$
$$ r_{g^{*}}(\overline{c})[g^{*}(b-\epsilon)-g^{*}(a+\epsilon)]=0. $$
From this, it must be that $g^{*}(x_2)=g^{*}(x_1)$. That is, $g^{*}$ is constant between $x_1$ and $x_2$, as was to be shown.

It remains to study the case where $r_{g^{*}}(x)>0$ for all $x\in(a,b)$.

Again, we apply Lemma \ref{lem:min_at_endpoints} and work with a sequence $\epsilon_{n}$ that tends to $0$ such that on $[a+\epsilon_{n},b-\epsilon_{n}]$, the function $r_{g^{*}}$ is non-constant and takes its minimum at $a+\epsilon_{n}$ or $b-\epsilon_{n}$.

As seen in the proof of the previous item, if $r_{g^{*}}$ takes its minimum at $a+\epsilon_{n}$, then

$$\int_{a+\epsilon_n}^{b-\epsilon_n}r_{g^{*}}dg^{*}=r_{g^{*}}(a+\epsilon_n)[g^{*}(b-\epsilon_n)-g^{*}(a+\epsilon_n)]$$

Analogously, if $r_{g^{*}}$ takes its minimum at $b-\epsilon_{n}$, it can be seen that

$$\int_{a+\epsilon_n}^{b-\epsilon_n}r_{g^{*}}dg^{*}=r_{g^{*}}(b-\epsilon_n)[g^{*}(b-\epsilon_n)-g^{*}(a+\epsilon_n)]$$

So in both cases it holds that

$$\int_{a+\epsilon_n}^{b-\epsilon_n}r_{g^{*}}dg^{*}=\min\{r_{g^{*}}(a+\epsilon_n),r_{g^{*}}(b-\epsilon_n)\}[g^{*}(b-\epsilon_n)-g^{*}(a+\epsilon_n)]$$

By taking the limit as $n\rightarrow\infty$, the left-side sequence of integrals converges to $\int_{a}^{b}r_{g^{*}}dg^{*}$ by Lebesgue's Dominated Convergence Theorem. 

On the right-hand side, since $g^*$ is non-decreasing on $[a,b]$, $[g^{*}(b-\epsilon_{n})-g^{*}(a+\epsilon_{n})]\leqslant[g^{*}(b^{-})-g^{*}(a^{+})]$. Also by the continuity of $r_{g^*}$ and the fact that $r_{g^*}(a)=r_{g^*}(b)=0$, we have:
$$\lim_{n\rightarrow\infty} \min\{r_{g^{*}}(a+\epsilon_n),r_{g^{*}}(b-\epsilon_n)\} = 0.$$

Since the right-hand side is the product of a sequence converging to zero and a bounded sequence, it yields:
$$\int_{a}^{b}r_{g^{*}}dg^{*}=0.$$

However, we assumed that $r_{g^{*}}(x)>0$ for all $x\in(a,b)$. Then $dg^{*}$ is identically zero on $(a,b)$. Thus $g^{*}$ must be constant (and therefore continuous) on $(a,b)$.
 
\end{itemize}
\end{proof}

\section{Continuity and Uniqueness Results}

We now show that the best non-decreasing approximant $g^*$ is continuous at every point in which $f$ is approximately continuous.

\begin{definition}[Approximate Continuity]
A point $x_0 \in (a,b)$ is a point of \textbf{approximate continuity} of a measurable function  $f$  if  for every $\delta > 0$, the set $A_{\delta} := \{x \in [a,b] : |f(x)- f(x_0)| < \delta \}$ has metric density 1 at $x_0$. That is, $\lim_{\mu(I)\to 0}\frac{\mu(A_{\delta}\cap I)}{\mu(I)}=1$ for any interval $I$ containing $x_0$.
\end{definition}

\begin{theorem}[Continuity of the Best Approximant]
\label{thm:continuity}
Let $x_0 \in (a,b)$ be a point of approximate continuity of $f$, and let one of the following conditions hold:
\begin{enumerate}[label=\alph*)]
    \item $\Phi$ is an $N$-function and $f\in L^\infty([a,b])$.
    \item $\Phi$ is an $N_{\infty}$-function and $f\in L^{\Phi}$.
\end{enumerate}
Then $x_0$ is a point of continuity of any best approximant $g^* \in \mu_\Phi(f/M_\Phi)$.
\end{theorem}

\begin{proof}
By Lemma \ref{lem:swetits_analogue}, we only need to consider the case where $r_{g^*}(x_0)=0$, since if $r_{g^*}(x_0)>0$, $g^*$ is already known to be constant, and thus continuous, in a neighborhood of $x_0$.

We proceed by contradiction. Assume $f(x_0) < g^*(x_0^+)$. We can then choose a $\delta > 0$ small enough such that $f(x_0)+\delta < g^*(x_0^+)$. Let $A_\delta = \{x \in [a,b] : |f(x)-f(x_0)| < \delta\}$. For any $x \in A_\delta$ near $x_0$, we have $f(x) < f(x_0)+\delta < g^*(x_0^+)$. Since $g^*$ is non-decreasing, $g^*(x_0^+) \le g^*(x)$ for $x > x_0$, thus $f(x) < g^*(x)$ and $sgn(f-g^*)(x)=-1$.

It is clear from the definition that
\begin{align*}
    r_{g^{*}}(x_{0}+\epsilon) &= \int_{a}^{x_{0}}\varphi(|f-g^{*}|)sgn(f-g^{*})d\mu+\int_{x_{0}}^{x_{0}+
    \epsilon}\varphi(|f-g^{*}|)sgn(f-g^{*})d\mu \\
    &= r_{g^{*}}(x_{0})+\int_{x_{0}}^{x_{0}+
    \epsilon}\varphi(|f-g^{*}|)sgn(f-g^{*})d\mu \\
    &= \int_{x_{0}}^{x_{0}+
    \epsilon}\varphi(|f-g^{*}|)sgn(f-g^{*})d\mu \ge 0.
\end{align*}

We can also see that $\varphi(|f(x_0) + \delta - g^{*}(x_0^+)|)\leq \varphi(|f(x)-g^{*}(x)|).$

Since $x_0$ is a point of approximate continuity from the right for $f$, then
$$\lim\limits_{\mu(I^{+})\rightarrow 0}\frac{\mu(A_{\delta}\cap I^{+})}{\mu(I^{+})}=1$$
for $I^{+}=(x_{0},x_{0}+\epsilon).$

Therefore,
\begin{align*}
    \int_{A_{\delta}\cap (x_{0},x_{0}+\epsilon)}\varphi(|f-g^{*}|)sgn(f-g^{*})\frac{d\mu}{\epsilon} \leq \\ -\varphi(|f(x_0) + \delta - g^{*}(x_{0}^{+})|)\frac{\mu(A_{\delta}\cap (x_{0},x_{0}+\epsilon))}{\epsilon}<0.
\end{align*}

Both $a)$ and $b)$ allow us to deduce that $\varphi(|f-g^{*}|)$ is bounded by some $M>0$, so we have

$$ \left|\int_{A_{\delta}^{c}\cap (x_{0},x_{0}+\epsilon)}\varphi(|f-g^{*}|)sgn(f-g^{*})\frac{d\mu}{\epsilon}\right|\leqslant M \frac{\mu(A_{\delta}^{c}\cap (x_{0},x_{0}+\epsilon))}{\epsilon}$$
Leaning again on the fact that $x_0$ is a point of approximate continuity from the right for $f$, we can see
$$ \int_{A_{\delta}^{c}\cap (x_{0},x_{0}+\epsilon)}\varphi(|f-g^{*}|)sgn(f-g^{*})\frac{d\mu}{\epsilon} \to 0, $$ as $\varepsilon$ goes to $0.$

Then for a sufficiently small $\epsilon$,
$$r_{g^{*}}(x_{0}+\epsilon)=\int_{x_{0}}^{x_{0}+ \epsilon}\varphi(|f-g^{*}|)sgn(f-g^{*})d\mu < 0.$$

But it was previously shown to be non-negative. This contradiction arose from the assumption that $f(x_0)<g^{*}(x_{0}^{+})$. Therefore, $f(x_0)\geq g^{*}(x_{0}^{+})$.

Assuming that $f(x_0)>g^{*}(x_{0}^{-})$, using approximate continuity from the left and proceeding analogously, we arrive at $f(x_0)\leq g^{*}(x_{0}^{-})$. Hence $f(x_0)=g^{*}(x_{0}^{-})=g^{*}(x_{0}^{+})$ and $g^{*}$ is continuous at $x_0$.
\end{proof}

We now present the main uniqueness result.

\begin{theorem}
 Let one of the following conditions hold
\begin{enumerate}[label=\alph*)]
    \item $\Phi$ is an $N$-function and $f\in L^\infty([a,b])$.
    \item $\Phi$ is an $N_{\infty}$-function and $f\in L^{\Phi}([a,b])$.
\end{enumerate}
If $f$ is approximately continuous at every point in $(a,b)$, then $\mu_\Phi(f/M_\Phi)$ is a singleton.
\end{theorem}
\begin{proof}

Let $g_1,g_2\in\mu_\Phi(f/M_\Phi)$. We shall name their average ${g^{*}} = \frac{(g_1+g_2)}{2}$. By the convexity of the function $\Phi$, we have:
$$ \int_a^b \Phi\left(\left|f-{g^{*}}\right|\right)d\mu \le \frac{1}{2}\int_a^b \Phi(|f-g_1|)d\mu + \frac{1}{2}\int_a^b \Phi(|f-g_2|)d\mu= \inf_{g \in C}\int_{a}^{b}\Phi(|f-g|)d\mu. $$
This implies that $g^{*}\in\mu_\Phi(f/M_\Phi)$ and the inequality above must be an equality.

Let us consider $S=\{x\in[a,b]: [(f-g_1)(f-g_2)](x)<0\}$. If we take $x\in S$, $|f-{g^{*}}|(x)<\frac{1}{2}|f-g_1|(x)+\frac{1}{2}|f-g_2|(x).$
Let us assume now that $\mu(S)>0$, then $\int_S \Phi\left(\left|f-{g^{*}}\right|\right)d\mu < \frac{1}{2}\int_a^b \Phi(|f-g_1|)d\mu + \frac{1}{2}\int_a^b \Phi(|f-g_2|)d\mu.$ This implies $\int_a^b \Phi\left(\left|f-{g^{*}}\right|\right)d\mu < \frac{1}{2}\int_a^b \Phi(|f-g_1|)d\mu + \frac{1}{2}\int_a^b \Phi(|f-g_2|)d\mu$, but this is a contradiction because $g_1,g_2\in\mu_\Phi(f/M_\Phi)$. It must be that $\mu(S)=0$.

Let us define the sets $\Omega_1 = \{x\in[a,b] : f(x) > {g^{*}}(x)\}$, $\Omega_2 = \{x \in[a,b] : f(x) < {g^{*}}(x)\}$, and $\Omega_3 = \{x\in[a,b]  : f(x) = {g^{*}}(x)\}$.

In $\Omega_3$, we have $2f=g_1+g_2$. That is $f-g_1=-(f-g_2)$, but we also know that $f-g_1$ and $f-g_2$ have the same sign almost everywhere. Therefore, $g_1=g_2$ almost everywhere in $\Omega_3$.

Now, let us assume that $x_0 \in \Omega_1$. We claim that $r_{g^{*}}(x_0) > 0$. Suppose, for contradiction, that $r_{g^{*}}(x_0)=0$.

Since $f(x_0)>{g^{*}}(x_0)$, we can take some $\delta>0$ small enough so that $f(x_0)>{g^{*}}(x_0)+\delta$. Since ${g^{*}}$ is continuous at $x_0$ by Theorem \ref{thm:continuity}, there is some $\epsilon>0$ such that ${g^{*}}(x_0)+\frac{\delta}{2}>{g^{*}}(x)$ for every $x\in(x_0 - \epsilon, x_0 + \epsilon)$. On the other hand, the approximate continuity of $f$ at $x_0$ yields that $A_{\frac{\delta}{2}}= \{x \in [a,b] : |f(x)-f(x_0)| < \frac{\delta}{2}\}$ has metric density $1$. Therefore, for $x\in A_{\frac{\delta}{2}}\cap (x_0 - \epsilon, x_0 + \epsilon)$, we can assert that $$f(x)>f(x_0)-\frac{\delta}{2}>{g^{*}}(x_0)+\frac{\delta}{2}>{g^{*}}(x).$$
In particular, $sgn(f-g^{*})(x)=1$ and $\varphi(|f(x_0)-\delta-g^{*}(x_0)|)<\varphi(|f(x)-g^{*}(x)|)$. Then,$$
    \int_{A_{\frac{\delta}{2}}\cap (x_{0}-\epsilon, x_{0})}\varphi(|f-g^{*}|)sgn(f-g^{*})\frac{d\mu}{\epsilon} \geq \varphi(|f(x_0) - \delta - g^{*}(x_{0})|)\frac{\mu(A_{\frac{\delta}{2}}\cap (x_{0}-\epsilon,x_{0}))}{\epsilon}>0.
$$
On the other hand, as we saw before $\int_{A_{\frac{\delta}{2}}^{c}\cap (x_{0}-\epsilon, x_{0})}\varphi(|f-g^{*}|)sgn(f-g^{*})\frac{d\mu}{\epsilon}\rightarrow 0$ as $\epsilon\rightarrow 0$ because $\varphi(|f-g^{*}|)$ is bounded.

So we have $\int_{x_{0}-\epsilon}^{x_{0}
    }\varphi(|f-g^{*}|)sgn(f-g^{*})d\mu >0$ for $\epsilon >0$ small enough. But since we assumed $r_{g^{*}}(x_0)=0$ this implies that
$ r_{g^{*}}(x_{0}-\epsilon)=\int_{a}^{x_{0}-\epsilon}\varphi(|f-g^{*}|)sgn(f-g^{*})d\mu<0$. This contradicts Lemma \ref{lem:swetits_analogue}. Therefore, $r_{g^{*}}(x_0)>0$ for all $x_0 \in \Omega_1$.

Again by Lemma \ref{lem:swetits_analogue}, if $r_{g^{*}}(x_0)>0$, then $g^{*}$ must be constant in a neighborhood of $x_0$. This implies that $g_1$ and $g_2$ are also constant in a neighborhood of any point in $\Omega_1$. A similar argument shows that $g_1$ and $g_2$ are also constant in a neighborhood of any point in $\Omega_2$. Since $g_1$ and $g_2$ are continuous on $(a,b)$ and are equal on the set $\Omega_3$, their local constancy on the open sets $\Omega_1$ and $\Omega_2$ forces them to be equal everywhere on $(a,b)$.

\end{proof}

On a final note, there is an interesting relationship between the modular best approximation problem we first described and the known Luxemburg norm, which is defined as:
$$ \|f\|_{\Phi}=\inf\left\{\lambda>0:\int_{a}^{b}\Phi\left(\frac{|f(x)|}{\lambda}\right)d\mu\leqslant1\right\}.$$
We denote by $\mu_{ \|\cdot \|_{\Phi}}(f/M_{\Phi})$ the set of best non-decreasing approximants in $M_{\Phi}$ for a given function $f\in L^{\Phi}$ with respect to the Luxemburg norm. Landers and Rogge proved in \cite{landersrogge} that for an $N$-function $\Phi$ and $f\in L^{\Phi}$, it holds that
$\mu_{ \|\cdot \|_{\Phi}}(f/M_{\Phi})=\delta \mu_{\Phi}(\frac{f}{\delta}/\frac{M_{\Phi}}{\delta})$
where $\delta:= \inf_{h\in M_{\Phi}}\|f-h\|_{\Phi}$.

\begin{corollary}
 Let one of the following conditions hold
\begin{enumerate}[label=\alph*)]
    \item $\Phi$ is an $N$-function and $f\in L^\infty([a,b])$.
    \item $\Phi$ is an $N_{\infty}$-function and $f\in L^{\Phi}([a,b])$.
\end{enumerate}
 If $x_0$ is a point of approximate continuity of $f\not\in M_{\Phi}$, then $x_0$ is a point of continuity of any $h^{*}\in\mu_{ \|\cdot \|_{\Phi}}(f/M_{\Phi})$.
\end{corollary}

\begin{proof}
The result is immediate from the continuity at $x_0$ of any $g^{*}\in\mu_{\Phi}(f/M_\Phi)$:
$$\mu_{ \|\cdot \|_{\Phi}}(f/M_{\Phi})=\delta \mu_{\Phi}(\frac{f}{\delta}/\frac{M_{\Phi}}{\delta})$$
where $\delta:= \inf_{h\in M_{\Phi}}\|f-h\|_{\Phi}$ is positive since we assumed $f\notin M_{\Phi}$.
\end{proof}

\begin{remark}
It is important to clarify the relationship between our results and the classical case of $L^1([0,1])$ approximation studied by Smith and Swetits \cite{smithswetits}. While their work is the main inspiration for our approach, the space $L^1([0,1])$ is not, strictly speaking, an Orlicz space generated by an $N_\infty$-function as defined in Section 2. The function $\Phi(t) = |t|$ has an associated derivative $\varphi(t)=1$ for $t>0$, which does not satisfy the condition $\varphi(0^{+})=0$. Consequently, the theorems presented in this paper should not be seen as direct extensions of the results for $L^1([0,1])$.
\end{remark}


\clearpage
    \noindent
    \begin{small}
        \textsc{1 Instituto de Matemática Aplicada San Luis, UNSL-CONICET and Departamento de Matemática, FCFMyN, UNSL, Av. Ejército de los Andes 950, 5700 San Luis, Argentina.} \\
        \textit{Email address}: \texttt{abenaven@unsl.edu.ar}

        \vspace{0.5cm}

        \textsc{2 Instituto de Investigaciones Matemáticas, UBA-CONICET and Departamento de Matemática, FCEyN, Universidad de Buenos Aires, Intendente Güiraldes 2160, 1428 Buenos Aires, Argentina.} \\
        \textit{Email address}: \texttt{costaponcejuan@gmail.com}

        \vspace{0.5cm}

        \textsc{3 Instituto de Matemática Aplicada San Luis, UNSL-CONICET and Departamento de Matemática, FCFMyN, UNSL, Av. Ejército de los Andes 950, 5700 San Luis, Argentina.} \\
        \textit{Email address}: \texttt{sfavier@unsl.edu.ar}
    \end{small}

\end{document}